\documentclass{amsart}

\usepackage{amsfonts, amssymb, amsmath, eucal, verbatim, amsthm, amscd, enumerate, mathtools}

\usepackage{graphicx}
\usepackage{framed}
\usepackage{tikz}
\usepackage{enumitem}
\usepackage{bbm}

\usepackage{ascmac}

\DeclareMathOperator*{\esssup}{ess\,sup}

\newtheorem{theorem}{Theorem}[section]
\newtheorem{lemma}[theorem]{Lemma}

\newtheorem{conjecture}[theorem]{Conjecture}

\newtheorem{corollary}[theorem]{Corollary}
\theoremstyle{definition}

\newtheorem{claim}[theorem]{Claim}

\theoremstyle{remark}
\newtheorem*{remark}{Remark}

\newcommand{\vertiii}[1]{{\left\vert\kern-0.25ex\left\vert\kern-0.25ex\left\vert #1 
\right\vert\kern-0.25ex\right\vert\kern-0.25ex\right\vert}}
\newcommand{\R}{{\mathbb R}}

\numberwithin{equation}{section}

\usepackage{setspace}

\def\1{\textbf{\rm 1}}

\def\XXint#1#2#3{{\setbox0=\hbox{$#1{#2#3}{\int}$}
\vcenter{\hbox{$#2#3$}}\kern-.5\wd0}}

\makeatletter
\@namedef{subjclassname@2020}{\textup{2020} Mathematics Subject Classification}
\makeatother

\begin{document}

\date{\today}


\author[Nakamura]{Shohei Nakamura}
\address[Shohei Nakamura]{Department of Mathematics, Graduate School of Science, Osaka University, Toyonaka, Osaka 560-0043, Japan}
\email{srmkn@math.sci.osaka-u.ac.jp}
\author[Tsuji]{Hiroshi Tsuji}
\address[Hiroshi Tsuji]{Department of Mathematics, Graduate School of Science, Osaka University, Toyonaka, Osaka 560-0043, Japan}
\email{tsuji@cr.math.sci.osaka-u.ac.jp}

\title[The functional  volume product under heat flow]{The functional  volume product under heat flow}

\begin{abstract}
We prove that the functional volume product for even functions is monotone increasing along the Fokker--Planck heat flow. 
This in particular yields a new proof of the functional Blaschke--Santal\'{o} inequality by K. Ball and also Artstein-Avidan--Klartag--Milman in the even case. 


This result is the consequence of a new understanding of the  regularizing property of the Ornstein--Uhlenbeck semigroup. 
That is, we establish an improvement of Borell's reverse hypercontractivity inequality for even functions and identify the sharp range of the admissible exponents. 
As another consequence of successfully identifying the sharp range for the inequality, we derive the sharp  $L^p$-$L^q$ inequality for the Laplace transform for even functions. The best constant of the inequality is attained by centered Gaussians, and thus this provides an analogous result to Beckner's sharp Hausdorff--Young inequality. 

Our technical novelty in the proof is the use of the Brascamp--Lieb inequality for log-concave measures and Cram\'{e}r--Rao's inequality in this context. 

\end{abstract}

\maketitle

\section{Introduction}
The celebrated Blaschke--Santal\'{o} inequality states that the volume product for a symmetric convex body is maximized by the Euclidean ball, or more generally ellipsoids. This inequality was upgraded to its functional form by K. Ball \cite{BallPhd} and  Artstein-Avidan--Klartag--Milman \cite{AKM},  and they proved that the functional volume product for a certain symmetric function is maximized by centered Gaussians. 
Given this fact, it is natural to expect some monotonicity property for the functional volume product under heat flow. 
We confirm this  phenomenon as a consequence of an investigation of the regularizing property of the Ornstein--Uhlenbeck semigroup. 
In more precise terms, we consider Borell's $L^p$-smoothing estimate (reverse hypercontractivity), and provide an improvement of it for even functions in terms of its critical  exponents (``Nelson's time condition"). 
We identify the sharp range of the admissible exponents for the improved reverse hypercontractivity, and thus provide a positive answer to the problem that emerged in our previous work \cite{NT}. 
Successfully identifying the sharp range for the inequality results yet another consequence to the Laplace transform. 
That is, we derive the sharp $L^p$-$L^q$ inequality for the Laplace transform for even functions. The best constant of the inequality is attained by centered Gaussians. 
Thus, this inequality may be seen as the analogue to Beckner's sharp Hausdorff--Young inequality \cite{Beck}, and confirms the ``detropicalised" version of the Blaschke--Santal\'{o} inequality that has been suggested by Tao in his blog post \cite{TaoBlog}. 

A fundamental bridge connecting these topics is our key identity  
$$
v(f)= \lim_{s\to0} c_s \big( \int_{\mathbb{R}^n} f\, dx \big)^{-\frac{q_s}{p_s}}\big\| P_s\big[ \big( \frac{f}{\gamma} \big)^\frac1{p_s} \big]\big\|_{L^{q_s}(\gamma)}^{q_s}. 
$$
Here, $v(f)$ denotes the functional volume product and $P_s$ is the Ornstein--Uhlenbeck semigroup. Also $c_s>0$ is some explicit constant and $p_s = 2s + O(s^2)$, $q_s = -2s + O(s^2)$; we give precise definitions below. 
This identity was implicitly observed in our previous work \cite{NT} and motivated by the vanishing viscosity argument  due to Bobkov--Gentil--Ledoux \cite{BGLJMPA}. 

\subsection{Heat flow monotonicity of the functional volume product}
The volume product of a symmetric convex body $K \subset \mathbb{R}^n$ is defined as $v(K) := |K||K^\circ|$ where $|\cdot|$ denotes the Euclidean volume, $K^\circ:=\{ x\in \mathbb{R}^n : \sup_{y\in K} \langle x, y\rangle \le 1\}$ is the polar body of $K$, and $\langle \cdot,\cdot\rangle$ is the natural inner product on $\mathbb{R}^n$. 
The volume product of a convex body plays a fundamental role in convex geometry and its theory is rich as it has links to numerous areas of mathematical sciences including algebraic topology, geometric analysis, geometry of numbers, harmonic analysis, probability and information theory, and systolic and symplectic geometry. 
We refer to the recent survey article by Fradelizi--Meyer--Zvavitch \cite{FMZSurvey} for historical background and recent developments. 
An innocent but far reaching question is what is the maximum and minimum of the volume product?  
The classical Blaschke--Santal\'{o} inequality gives an answer to the maximum and states that $v(K)\le v(\mathbf{B}^n_2)$ for all symmetric convex bodies $K$, where $\mathbf{B}^n_p:=\{x\in \mathbb{R}^n: (\sum_{i=1}^n |x_i|^p)^{1/p} \le 1\}$ denotes the unit $\ell^p$-ball for $p\in [1,\infty]$. 
This inequality was proved by Blaschke \cite{Blaschke} for $n=2,3$ and Santal\'{o} \cite{Santalo} for $n\ge4$. We refer to \cite{BK,MeyPaBook,MeyPa,MR19,Saint} for several alternative proofs. 
On the other hand, to identify the minimum of the volume product among symmetric convex bodies, known as Mahler's conjecture, is still an open problem, and has been for almost a century. Mahler expected that the minimum is attained by the Euclidean cube $\mathbf{B}^n_\infty$ and confirmed it when $n=2$ \cite{Mar1,Mar2}. 
A recent breakthrough was brought by Iriyeh--Shibata \cite{IriShi} which proved Mahler's conjecture when $n=3$, and their proof was significantly simplified by Fradelizi et al. \cite{FHMRZ}. The problem for $n\ge4$ is open despite several partial answers; see the survey article \cite{FMZSurvey}. 

It was observed by Ball \cite{BallPhd} that several geometrical inequalities and problems regarding the volume of convex bodies may be formulated in terms of log-concave functions, and that this functional upgrading sheds new light on the original  geometrical problems.  
Following this idea, the Blaschke--Santal\'{o} inequality was also extended to its functional form by Ball \cite{BallPhd} and Artstein-Avidan--Klartag--Milman \cite{AKM}, see also Fradelizi--Meyer \cite{FraMeyMathZ} and Lehec \cite{LehecDirect,LehecYaoYao} for further generalizations as well as alternative proofs. 
For a nonnegative function $f$ on $\mathbb{R}^n$, its polar function, denoted by $f^\circ$, is defined as 
$$
f^\circ(x) := \inf_{y\in\mathbb{R}^n} \frac{e^{-\langle x,y\rangle}}{f(y)},\;\;\; x\in \mathbb{R}^n. 
$$
We often identify $f = e^{-\phi}$ for some $\phi: \mathbb{R}^n\to \mathbb{R}\cup\{+\infty\}$ and say that $f$ is log-concave if $\phi$ is convex on $\{ \phi < +\infty \}$. 
 In this terminology, $f^\circ(x) = e^{-\phi^*(x)}$ holds true where $\phi^*(x):= \sup_{y\in \mathbb{R}^n} [\langle x,y\rangle - \phi(y)]$ is the Legendre transform of $\phi$.
The functional volume product for $f$ is defined as 
$$
v(f):= \int_{\mathbb{R}^n} f\, dx \int_{\mathbb{R}^n} f^\circ\, dx. 
$$
For a symmetric convex body $K \subset \mathbb{R}^n$, the Minkowski functional $\|x\|_K:= \inf\{r>0: x\in rK\}$, $x\in\mathbb{R}^n$,  becomes a norm on $\mathbb{R}^n$ and satisfies 
\begin{equation}\label{e:Func->Geo}
\int_{\mathbb{R}^n} e^{-\frac12 \|x\|_K^2}\, dx = \frac{(2\pi)^\frac{n}2}{|\mathbf{B}^n_2|}|K|,
\;\;\;
\big( \frac12 \|\cdot\|_K^2 \big)^*(x)
= 
\frac12 \|x\|_{K^\circ}^2.
\end{equation}
It is clear from these properties that the standard  Gaussian $\gamma(x):= (2\pi)^{-\frac{n}2} e^{ -\frac12 |x|^2 }$ plays the role of $\mathbf{B}^n_2$ in this functional formulation. 
More generally, for a positive definite matrix $A$, we denote the centered Gaussian with covariance matrix $A$ by $\gamma_A(x):= {\rm det}\, (2\pi A)^{-\frac12} e^{ -\frac12 \langle x,A^{-1}x\rangle }$. 
Then the functional Blaschke--Santal\'{o} inequality states the following. 
\begin{theorem}[Ball \cite{BallPhd}, Artstein-Avidan--Klartag--Milman \cite{AKM}]\label{t:FBS}
For all even functions $f\colon \mathbb{R}^n \to \R_+$ with $0<\int_{\mathbb{R}^n}f\, dx <+\infty$, 
 \begin{equation}\label{e:FBS}
    v(f)\le v(\gamma) = (2\pi)^n. 
 \end{equation}
 The case of equality in \eqref{e:FBS} appears if and only if $f = c \gamma_A$ for some positive definite matrix $A$ and $c>0$. 
\end{theorem}
By choosing $f= e^{-\frac12 \|x\|_K^2}$, \eqref{e:FBS} rederives  the classical Blaschke--Santal\'{o} inequality  since we have  $v( e^{-\frac12 \|\cdot\|_K^2} ) = {(2\pi)^n}{|\mathbf{B}^n_2|^{-2}} v(K)$ from \eqref{e:Func->Geo}. 
We note that the evenness assumption was weakened to the condition that the barycenter of $f$ is 0 in \cite{AKM,LehecDirect,LehecYaoYao}. 

Given the inequality \eqref{e:FBS}, a natural question emerges: is the functional volume product monotone increasing along some heat flow? We give a positive answer to the question in this paper. Regarding this purpose, there are at least two reasons to expect such a monotonicity statement. The first is of course about the extremizers of \eqref{e:FBS} which are the centered Gaussians. Perhaps the most famous example of this type of inequality is the fact that the Shannon entropy is maximized by the standard Gaussian among isotropic random variables. This is a consequence of Boltzmann's H-theorem, which states that the Shannon entropy is monotone increasing along heat flow.  
We mention works by Artstein-Avidan et al. \cite{AKSW}, Fathi \cite{Fathi}, and the second author \cite{Tsuji}  for entropic interpretations of the functional Blaschke--Santal\'{o} inequality.
The second reason is an observation made in \cite{AKM} that the functional volume product is monotone increasing under the application of the Steiner symmetrization if the input function is even. 
Although this fact does not give a direct proof of \eqref{e:FBS}, it reduces the matter to the case $n=1$. In this regard, the heat flow monotonicity of the functional volume product was implicitly suggested in \cite{AKM}. 

It turns out that an appropriate flow for our purpose is the Fokker--Planck heat flow. For a nonnegative initial data $f_0 \in L^1(dx)$, let $f_t$ be a solution to the Fokker--Planck equation 
$$
\partial_t f_t = \mathcal{L}^* f_t := \Delta f_t + {\rm div}\, (xf_t). 
$$
It is well-known that the Fokker--Planck heat flow is the dual of the Ornstein--Uhlenbeck semigroup defined as 
\begin{equation}\label{e:DefOU}
P_tg(x)
:=
\int_{\mathbb{R}^n} e^{ -\frac{| e^{-t}x - y |^2}{2(1-e^{-2t})} } g(y)\, \frac{dy}{ ( 2\pi (1-e^{-2t}) )^\frac{n}{2} }
\;\;\; (x,t) \in \mathbb{R}^n\times [0,\infty). 
\end{equation}
Namely, the solution $f_t$ has an explicit representation as 
$$
f_t(x) = P_t^*f_0(x) = \big( e^{nt} f_0(e^t\cdot) \big) \ast \gamma_{ 1-e^{-2t}}(x) = 
\int_{\mathbb{R}^n} e^{ -\frac{| x - e^{-t}y |^2}{2(1-e^{-2t})} } \, \frac{f_0( y )dy}{ ( 2\pi (1-e^{-2t}) )^\frac{n}{2} }, 
$$
where we denote $\gamma_\beta:= \gamma_{\beta {\rm id}_{\mathbb{R}^n}}$ for $\beta>0$ and  $P_t^*$ means the dual of $P_t$ with respect to the $L^2(dx)$-inner product. 
Note that, for a nonnegative initial data $g \in L^1(d\gamma)$, $u_t = P_tg$ solves the heat equation $\partial_t u_t = \mathcal{L} u_t:= \Delta u_t  - x\cdot \nabla u_t$, $u_0 = g$. 

Our first result is to confirm that the functional volume product is monotone increasing along the Fokker--Planck heat flow.
\begin{theorem}\label{t:MonoVP}
	For any nonzero and nonnegative even function $f_0\in L^1(dx)$,  it holds that 
 $$v(f_{t_1})\le v(f_{t_2})$$ for any $0\le t_1 \le t_2$,  where  $f_t$ is the solution to $\partial_t f_t = \mathcal{L}^* f_t$ with the initial data $f_0$. 
\end{theorem}
This monotonicity immediately yields the functional Blaschke--Santal\'{o} inequality \eqref{e:FBS} for even functions since $\lim_{t\to\infty} f_t = (\int_{\mathbb{R}^n} f_0\, dx) \gamma$ and in this sense we give an alternative proof of it. We emphasize that our proof is completely geometry free and has a different nature from the proofs in \cite{AKM,BallPhd,LehecDirect,LehecYaoYao}. 

\subsection{An improvement of Borell's reverse hypercontractivity}
Theorem \ref{t:MonoVP} is a consequence of our new approach to the study of $v(f)$ which is based on the regularizing property of the Ornstein--Uhlenbeck semigroup. 
Let us explain how this regularizing property relates to the functional volume product $v(f)$. 
One obvious difficulty to establishing the flow monotonicity of the functional volume product (Theorem \ref{t:MonoVP}) lies in the definition of the polar function, which involves the infimum. 
The standard strategy to establish the heat flow monotonicity of some functional, given by an integral, is to take a time derivative of the functional and then look for some appropriate representation of it by applying  integration by parts. However the presence of the infimum  prevents one from  computing the time derivative and  integrating by parts. 
Our new idea to overcome this difficulty is to  regard  \eqref{e:FBS} as a limiting case of some improvement of Borell's reverse hypercontractivity. 
A similar idea may be found in the work of  Brascamp--Lieb \cite{BraLi_Adv}, where they rederived the Pr\'{e}kopa--Leindler inequality as a limiting case of the sharp reverse Young convolution inequality. 
Successfully applying this idea, which originated in our previous work \cite{NT}, is the main methodological novelty of this paper. 
Precisely, this new idea is represented by the following simple observation which reveals a link between the functional volume product and  the Brascamp--Lieb multilinear inequality  \cite{Barthe1,BW,BCCT,BN,BraLi_Adv,CDP,CouLiu,Lieb}. 
\begin{lemma}\label{Prop:Obs1}
For each small $s>0$, let $q_s<0<p_s$ and a symmetric matrix $\mathcal{Q}_s$ be such that\footnote{For $p\in\mathbb{R}\setminus\{0\}$, $p' := \frac{p}{p-1}\in \mathbb{R}\setminus\{0\}$ denotes the usual H\"{o}lder conjugate.} $p_s \to 0$, $\frac{p_s}{q_s'} \to 1$ and $p_s \mathcal{Q}_s \to \frac1{2\pi}  \begin{pmatrix} 
0 & - {\rm id}_{\mathbb{R}^n} \\
 - {\rm id}_{\mathbb{R}^n} & 0
 \end{pmatrix}$
as $s\to 0$. 
Then for continuous  $f_1,f_2:\mathbb{R}^n\to\mathbb{R}_+$ that have Gaussian decay, we have that  
\begin{equation}\label{e:Obs1}
\lim_{s\downarrow0} \bigg(
\int_{\mathbb{R}^{2n}} 
e^{ -\pi \langle x,\mathcal{Q}_s x\rangle }
f_1(x_1)^{\frac1{p_s}} f_2(x_2)^{\frac1{q_s'}} \, dx \bigg)^{p_s}
= 
\sup_{x \in \mathbb{R}^n} f_1(x) f_2^\circ(x)^{-1}.
\end{equation}
In particular, for $f_1=f^\circ$ and $f_2=f$, it holds that 
\begin{equation}\label{e:Obs2}
\lim_{s\downarrow0} \bigg(
\int_{\mathbb{R}^{2n}} 
e^{ -\pi \langle x,\mathcal{Q}_s x\rangle }
f^\circ (x_1)^{\frac1{p_s}} f(x_2)^{\frac1{q_s'}} \, dx \bigg)^{p_s}
= 1. 
\end{equation}
\end{lemma}

The proof of this lemma is a simple limiting argument: 
\begin{align*}
&\bigg(
\int_{\mathbb{R}^{2n}} 
e^{ -\pi \langle x,\mathcal{Q}_s x\rangle }
f_1(x_1)^{\frac1{p_s}} f_2(x_2)^{\frac1{q_s'}} \, dx \bigg)^{p_s}\\
&=
\big\|e^{-\pi  \langle x,p_{s}\mathcal{Q}_{s} x \rangle } f_1(x_1) f_2(x_2)^{\frac{p_{s}}{q_s'}}\big\|_{L^{\frac1{p_s}}(\mathbb{R}^{2n}, dx)}
\to 
\sup_{x_1 \in \mathbb{R}^n} f_1(x_1) f_2^\circ(x_1)^{-1}. 
\end{align*}
A typical example of $p_s,q_s,\mathcal{Q}_s$ satisfying the condition is 
\begin{align}\label{e:ExampleBLdata}
&p_s = 2s + O(s^2),\; q_s = -2s + O(s^2),\\
&\mathcal{Q}_s := \frac1{2\pi ( 1-e^{-2s} )} 
\begin{pmatrix} 
 (1- \frac{1-e^{-2s}}{p_s}){\rm id}_{\mathbb{R}^n} & -e^{-s} {\rm id}_{\mathbb{R}^n} \\
 -e^{-s} {\rm id}_{\mathbb{R}^n} & e^{-2s}(1 - \frac{1-e^{2s}}{q_s}) {\rm id}_{\mathbb{R}^n}
\end{pmatrix}.\nonumber
\end{align}
From Lemma \ref{Prop:Obs1} it is immediate to deduce a bound on the functional volume product from  the Brascamp--Lieb type inequality as follows. 
For each $s>0$, fix $p_s,q_s,\mathcal{Q}_s$ satisfying the assumptions in Lemma \ref{Prop:Obs1}. 
Let ${\rm BL}_s^{(e)}\ge0$ be the largest constant for which the inequality 
\begin{equation}\label{e:HC-IBL}
\int_{\mathbb{R}^{2n}} e^{ -\pi \langle x,\mathcal{Q}_s x\rangle }\prod_{i=1,2} f_i(x_i)^{c_i(s)} \, dx 
\ge {\rm BL}_s^{(e)} 
\prod_{i=1,2} \big( \int_{\mathbb{R}^n} f_i\, dx_i \big)^{c_i(s)} 
\end{equation}
holds for all nonnegative \textit{even} functions $f_i \in L^1(dx)$, where $c_1(s):= \frac{1}{p_s}$ and $c_2(s):=\frac1{q_s'}$. 
Then \eqref{e:Obs2} yields that 
\begin{equation}\label{e:BL-VP}
\liminf_{s\to0} \big({\rm BL}_s^{(e)}\big)^{-p_s} \ge  v(f)
\end{equation}
for all nonnegative even functions $f$,  and hence the functional Blaschke--Santal\'{o} inequality \eqref{e:FBS} would follow if one could identify ${\rm BL}_s^{(e)}$ for each $s>0$, and prove that $\liminf_{s\to0} \big( {\rm BL}_s^{(e)}\big)^{-p_s}=(2\pi)^n$.
Given these observations, it is natural to study the above Brascamp--Lieb type inequality. 
In particular, as Lieb's fundamental theorem \cite{Lieb} suggests, one may  expect that ${\rm BL}_s^{(e)}$ is exhausted by centered Gaussians. In fact, \eqref{e:HC-IBL} may be regarded as an example of the inverse Brascamp--Lieb inequality which is systematically investigated\footnote{We emphasize that Barthe--Wolff considered all nonnegative inputs $f_i$ rather than even functions.} by Barthe--Wolff \cite{BW}.  We also refer  to the related works \cite{BN,CDP,CouLiu}.  
However, the crucial point here is that the Brascamp--Lieb data for  \eqref{e:HC-IBL} does not satisfy Barthe--Wolff's non-degeneracy condition and hence one cannot appeal to the general theory in \cite{BW}. 
This point strongly motivates us to go beyond Barthe--Wolff's non-degeneracy  condition by assuming the evenness on the inputs. 
We refer the reader to Section \ref{SecRemark} as well as our previous work \cite{NT} for more detailed discussion on this problem. In the following, we focus on a specific choice of $(p_s,q_s, \mathcal{Q}_s)$, namely \eqref{e:ExampleBLdata}. 
\begin{remark}
    After we uploaded the first version of this paper, Cordero-Erausquin pointed out to us about the discussion of Klartag and Tao in Tao's blog post \cite{TaoBlog}. 
    In there, a Laplace transform formulation of Blaschke--Santal\'{o} inequality and Mahler's conjecture (``detropicalised" version) have been proposed. 
    Tao gave a limiting argument (``tropical limit") that connects the Laplace transform and the functional volume product. 
    This limiting argument coincides with Lemma \ref{Prop:Obs1} with a choice $p_s=1-e^{-2s}$ and $q_s = 1-e^{2s}$, see also forthcoming Corollary \ref{Cor:Laplace}. 
\end{remark}

It is worth emphasizing that there is a lot of candidates for $p_s,q_s,\mathcal{Q}_s$ satisfying the conditions of Lemma \ref{Prop:Obs1} other than the specific form of \eqref{e:ExampleBLdata}. 
However, as we will explain below, we are guided to the above specific form of $p_s,q_s,\mathcal{Q}_s$ by the nature of the Ornstein--Uhlenbeck semigroup, and that we identify \eqref{e:ExampleBLdata} is a key to our proof of Theorem \ref{t:MonoVP}. 
A benefit of the choice \eqref{e:ExampleBLdata} is represented by the following identity; see \cite{BW,NT} for instance. For a given $f_0$, we have that 
$$
C_s\int_{\mathbb{R}^{2n}} e^{ -\pi \langle x,\mathcal{Q}_s x\rangle } f_1(x_1)^{\frac1{p_s}} f_2(x_2)^{\frac1{q_s'}} \, dx 
= 
\big\| P_s \big[ \big(\frac{f_0}{\gamma}\big)^{\frac1{p_s}} \big] \big\|_{L^{q_s}(\gamma)}, 
$$
where $f_1:= f_0$ and 
\begin{equation}\label{e:TransHC-IBL}
 f_2:= \big\| P_s \big[ \big(\frac{f_0}{\gamma}\big)^{\frac1{p_s}} \big] \big\|_{L^{q_s}(\gamma)}^{-q_s} P_s \big[ \big(\frac{f_0}{\gamma}\big)^{\frac1{p_s}} \big]^{q_s} \gamma,
\;\;\; C_s:= \big( \frac{ (2\pi)^{ \frac12( \frac{1}{p_s} + \frac{1}{q_s'}) -1 } }{ \sqrt{1-e^{-2s}} } \big)^n. 
\end{equation}
By virtue of this identity and a duality argument, \eqref{e:HC-IBL} is equivalent to the inequality 
\begin{equation}\label{e:RevHC_0}
    \big\| P_s \big[ \big(\frac{f_0}{\gamma}\big)^{\frac1{p}} \big] \big\|_{L^{q}(\gamma)}
    \ge 
    {\rm H}_{s,p,q}^{(e)} \big( \int_{\mathbb{R}^n} \frac{f_0}{\gamma}\, d\gamma \big)^\frac1{p}
\end{equation}
for all nonnegative even functions $f_0 \in L^1(dx)$ with $p=p_s$, $q=q_s$ and ${\rm H}_{s,p,q}^{(e)} = C_s {\rm BL}_s^{(e)}$. 
This inequality is reminiscent of Borell's reverse hypercontractivity by identifying $g = \frac{f_0}{\gamma}$. 
\begin{theorem}[Borell \cite{Borell}]\label{t:Borell}
	Let $s>0$ and $q<0<p <1$ satisfy $q\ge q(s,p) := 1 + e^{2s}(p-1)$. Then it holds that 
	\begin{equation}\label{e:RevHC}
	\big\| P_s \big[ g^\frac1{p} \big] \big\|_{L^{q}(\gamma)} 
	\ge 
	\big( \int_{\mathbb{R}^n} g\, d\gamma \big)^\frac1{p}
	\end{equation}
	for all nonnegative $g\in L^1(\gamma)$. 
	Moreover, $q(s,p)$ is the sharp threshold in the sense that 
	\begin{equation}\label{e:NecNelson}
	q < q(s,p)\; \Rightarrow\; \inf_{\beta>0,a\in \mathbb{R}^n} \big\| P_s \big[ \big( \frac{\gamma_\beta(\cdot +a)}{\gamma} \big)^\frac1{p} \big] \big\|_{L^{q}(\gamma)} 
	=
	0. 
	\end{equation}
\end{theorem}
The necessary condition $q \ge q(s,p)$ may be written as 
$
\frac{q-1}{p-1}\le e^{2s}
$
and is often referred to Nelson's time condition. 
Since the $L^q$-norm for $q<0$ measures the positivity of a function, the inequality \eqref{e:RevHC} describes the regularizing effect\footnote{In Borell's paper \cite{Borell}, it is called as a positivity improving.} of the Ornstein--Uhlenbeck semigroup in a quantitative way. 
In particular, \eqref{e:RevHC} for smaller $q\ll0$ manifests a stronger   regularizing effect of $P_s$, and $q(s,p)$ provides a limitation of the regularization. 

Let us go back to our problem \eqref{e:RevHC_0} in which case $p=p_s = 2s + O(s^2)$ and $q=q_s = -2s +O(s^2)$. 
Since $q(s,p_s) = O(s^2) \gg q_s$, our inequality \eqref{e:RevHC_0} is beyond the Nelson's time and so one cannot directly apply Theorem \ref{t:Borell}. Moreover in view of \eqref{e:NecNelson}, one cannot expect any non-trivial inequality  \eqref{e:RevHC_0}  that holds  for \textit{all} $f_0\in L^1(dx)$, namely the symmetry of $f_0$ is essential. This feature is consistent with the functional Blaschke--Santal\'{o} inequality; one cannot expect  \eqref{e:FBS} without any symmetry on $f$. 
These observations suggest that the regularizing effect of $P_s$ may be improved if the initial data has a symmetry. Furthermore, this heuristic may be  quantified in terms of the range of $p,q$ for which \eqref{e:RevHC_0} holds with ${\rm H}_{s,p,q}^{(e)}>0$. Given these observations, it is natural to ask the following questions. What is the largest range of generic $q<0<p$ for which \eqref{e:RevHC_0} holds with ${\rm H}_{s,p,q}^{(e)}>0$ for each $s>0$? 
If ${\rm H}_{s,p,q}^{(e)}>0$ then what is the largest value of ${\rm H}_{s,p,q}^{(e)} $? 
These problems were formulated in our previous work \cite{NT}, where we obtained some partial progress as follows. 
Towards the necessary range of $p,q$, we observed in \cite[(1.21)]{NT} that for $q<0<p$,  
\begin{equation}\label{e:GaussRevHC}
\inf_{\beta>0} \big\| P_s \big[ \big(\frac{\gamma_\beta}{\gamma}\big)^\frac1{p} \big] \big\|_{L^{q}(\gamma)}>0
\;\;\;
\Leftrightarrow
\;\;\; 
1-e^{2s}\le q < 0 < p \le 1-e^{-2s}.
\end{equation}
On the other hand, towards the sufficient  direction, we proved  in \cite[Theorem 1.7]{NT} that \eqref{e:RevHC_0} holds true with ${\rm H}_{s,p,q}^{(e)} = 1$ in the partial range $0<p\le 1-e^{-2s}$ and $q \ge -p$. 
The proof in \cite{NT}  is based on a combination of Harnack's inequality and Lehec's argument \cite{LehecYaoYao} using the multiplicative Pr\'{e}kopa--Leindler inequality together with the Yao--Yao equipartition theorem.  
To our best knowledge, the partial range $p\le 1-e^{-2s}$ and $q\ge -p$ seems to be the best possible as long as one utilizes Lehec's argument. 
Our new alternative route is, as Theorem \ref{t:MonoVP} suggests, the flow monotonicity and this leads us to the following complete answer. 
\begin{theorem}\label{t:SymmRevHC}
    Let $s>0$ and $1-e^{2s}\le q < 0 < p \le 1-e^{-2s}$. Then \eqref{e:RevHC_0} holds for all nonnegative even functions  $f_0 \in L^1(dx)$ with ${\rm H}_{s,p,q}^{(e)} = 1$. Equality is established when $f_0 = \gamma$. 
    Moreover, the range $1-e^{2s}\le q < 0 < p \le 1-e^{-2s}$ is the best possible in the sense of \eqref{e:GaussRevHC}. 
\end{theorem}

It is crucial that we manage to identify the sharp range of the improved reverse hypercontractivity. Indeed, at the endpoint $p=1-e^{-2s}$ and $q=1-e^{2s}$, Theorem \ref{t:SymmRevHC} contains a further consequence. 
Let us define the Laplace transform $\mathfrak{L}$ by 
$$
\mathfrak{L}f(x) := \int_{\mathbb{R}^n} e^{\langle x,z\rangle} f(z)\, dz,\quad x\in\mathbb{R}^n, 
$$
for a nonnegative function $f$; if the integral does not converge we regard $\mathfrak{L}f(x)=\infty$. We also  use a convention $\infty^{-1} = 0$. 
The Laplace transform naturally appears and has been used in convex geometry. We refer to the work of Klartag--Milman \cite{KlaMil12JFA} and references therein. 
As a corollary of Theorem \ref{t:SymmRevHC}, we obtain the sharp inequality for the Laplace transform. 
\begin{corollary}\label{Cor:Laplace}
    Let $p\in (0,1)$ and $q=p'<0$. Then it holds that 
    $$
\big\| \mathfrak{L} f \big\|_{L^{q}(dx)} \ge \frac{\big\| \mathfrak{L}  \gamma  \big\|_{L^q(dx)} }{ \| \gamma\|_{L^p(dx)} } \| f\|_{L^p(dx)}, 
$$
   for any nonnegative and even $f \in L^p(dx)$. 
\end{corollary}
This inequality is the analogue to Beckner's sharp Hausdorff--Young inequality. 
In particular, Corollary \ref{Cor:Laplace} confirms the ``detropicalised" version of the Blaschke--Santal\'{o} inequality that has been suggested by Tao in his blog post \cite{TaoBlog}. 
We refer the identity \eqref{e:EquivForm} in Section 2 to derive Corollary \ref{Cor:Laplace} from Theorem \ref{t:SymmRevHC}. 

Our approach based on the new viewpoint of hypercontractivity has a nature-based interpretation and clarifies internal relations of several subjects that we explained. For further links to recent works of Berndtsson--Mastrantonis--Rubinstein \cite{BMR} and Kolesnikov--Werner \cite{KW}, we refer to Section \ref{SecRemark}. 
There are other positive consequences of this viewpoint. 
For instance, we have obtained the forward type $L^p$-$L^q$ inequality for the Laplace transform: $ \big\| \mathfrak{L} f \big\|_{L^{q}(dx)} \le C \| f\|_{L^p(dx)} $ for $q<0<p$, under the uniform log-concavity assumption in \cite{NT}. This in particular yields a quantitative lower bound of the volume product for convex bodies in terms of the curvature of the boundary. 
In the forthcoming paper, we will develop our argument and improve the stability estimate for the functional Blaschke--Santal\'{o} inequality due to Barthe--B\"{o}r\"{o}czky--Fradelizi \cite{BBF} under the uniform log-concavity assumption.

In Section \ref{Sec2}, we will derive Theorems \ref{t:MonoVP} and \ref{t:SymmRevHC} from a flow monotonicity of some functional associated with  hypercontractivity, see the forthcoming Theorem \ref{t:MonoHC}. 
In Section \ref{SecRemark}, we provide discussions about the relation of our work to recent works of Berndtsson--Mastrantonis--Rubinstein \cite{BMR} as well as Kolesnikov--Werner \cite{KW}. 



\section{Proof of Theorems \ref{t:MonoVP} and \ref{t:SymmRevHC}}\label{Sec2}
We prove a stronger monotonicity statement that yields Theorems \ref{t:MonoVP} and \ref{t:SymmRevHC} at the same time. 
\begin{theorem}\label{t:MonoHC}
Let $s>0$, $p=1-e^{-2s}$ and $q=1-e^{2s}$. Then for any even function $f_0:\mathbb{R}^n\to \mathbb{R}_+$ with $0 < \int_{\R^n}f_0\, dx < +\infty$, 
$$
[0,\infty)\ni t \mapsto  \big\| P_s\big[ \big(\frac{f_t}{\gamma} \big)^\frac1p \big] \big\|_{L^q(\gamma)}^q 
$$
is monotone increasing on $t\in[0,\infty)$, where $f_t:= P_t^* f_0$. 
\end{theorem}
Two remarks about this theorem are in order. 
The first is that such a flow monotonicity scheme may be found in the work of Aoki et al. \cite{ABBMMS} under the Nelson's time condition.
In fact, for $p,q$ satisfying the Nelson's time condition, the above monotonicity has been already proved in \cite{ABBMMS}. In this sense, Theorem \ref{t:MonoHC} improves the work of \cite{ABBMMS} under the evenness of the initial data. 
The second remark is about the speciality of the exponents $p,q$ in the above. Notice that the strongest, and hence the most difficult to prove, inequality in Theorem \ref{t:SymmRevHC} appears at the endpoint $(p,q)=(1-e^{-2s}, 1-e^{2s})$ as other cases follow from H\"{o}lder's inequality. 
Nevertheless, the endpoint case has a special character that makes things canonical and we will appeal to the virtue of it. 
Firstly, the inequality \eqref{e:RevHC_0} becomes linear invariant only at the endpoint $(p,q)=(1-e^{-2s}, 1-e^{2s})$. 
Secondly, we may check from \eqref{e:DefOU} that 
\begin{align*}
&C_{s,p,q}^{-q}\big\| P_s\big[ \big( \frac{h}{\gamma} \big)^\frac1p \big] \big\|_{L^q(\gamma)}^q \\
&=
\int_{\mathbb{R}^n}
\bigg(
\int_{\mathbb{R}^n} 
e^{ \frac{e^{-s}}{1-e^{-2s}}\langle x,y\rangle } h(y)^\frac1p 
e^{ -\frac12( \frac1{1-e^{-2s}} -\frac1p )|y|^2 }\, dy
\bigg)^q e^{ -\frac12( 1 - \frac{q}{1-e^{2s}} )|x|^2 }\, dx 
\end{align*}
for genuine $p,q\in\mathbb{R}\setminus\{0\}$ where $C_{s,p,q}:=\big( \frac{ (2\pi)^{ \frac12(\frac1p+\frac1{q'}) -1 } }{ \sqrt{1-e^{-2s}} } \big)^{n}$.
In particular, when $(p,q)=(1-e^{-2s},1-e^{2s})$, the above expression becomes  simpler as 
\begin{equation}\label{e:EquivForm}
    \big\| P_s\big[ \big(\frac{f_t}{\gamma} \big)^\frac1p \big] \big\|_{L^q(\gamma)}^q
    =
    C_s^q \int_{\mathbb{R}^n} F_t(e^{-s}x)^q\, dx
    =
    C_s^q e^{ns} \int_{\mathbb{R}^n} F_t(x)^q\, dx, 
\end{equation}
where $C_s$ is defined in \eqref{e:TransHC-IBL} and 
\begin{equation}\label{e:F_t}
    F_t(x) = F_t^{(s)}(x) := \int_{\mathbb{R}^n} e^{ \frac{1}{p}\langle x,z\rangle}  f_t(z)^\frac1p\, dz= \mathfrak{L}[f_t^\frac1p](\frac{x}{p}).  
\end{equation}
Hence, Theorem \ref{t:MonoHC} is equivalent to a monotonicity of 
$$
Q_s(t):= \log\, \int_{\mathbb{R}^n} F_t(x)^q\, dx. 
$$
Before proving Theorem \ref{t:MonoHC}, let us first complete proofs of Theorems \ref{t:MonoVP} and \ref{t:SymmRevHC} by assuming it. 
\begin{proof}[Proof of Theorem \ref{t:MonoVP}]
Without loss of generality, we may assume $\int_{\mathbb{R}^n} f_t\, dx = 1$ and hence $v(f_t) = \int_{\mathbb{R}^n} f_t^\circ\, dx$ since $\int_{\mathbb{R}^n} f_t\, dx = \int_{\mathbb{R}^n} f_0\, dx$. 
Let us take arbitrary $0\le t_1 < t_2$ and show that 
$$
v(f_{t_1}) \le v(f_{t_2}). 
$$
Theorem \ref{t:MonoHC} implies that $Q_s(t_1)\le Q_s(t_2)$. 
On the other hand, $f_t$ is continuous for each fixed $t>0$ and hence we have that\footnote{To be precise, we need to check that $F_t^{(s)}(x) <+\infty$ for all sufficiently small $s>0$ by fixing arbitrary  $t>0$ and  $x\in\mathbb{R}^n$  but this follows from the fact that $f_t$ has the Gaussian decay.} 
\begin{equation}\label{e:CreatPolar}
\lim_{s\to 0} F_t^{(s)}(x)^q = \big( \esssup_{z\in \mathbb{R}^n} e^{\langle x,z\rangle} f_t(z) \big)^{-1} = f_t^\circ(x)
\end{equation}
for all $x\in\mathbb{R}^n$, by virtue of $p = 1-e^{-2s}$ and $q=1-e^{2s}$. 
When $t=0$, $f_0$ may not be continuous nor fast decaying, but we still have that 
$
\liminf_{s\to 0} \int_{\mathbb{R}^n} F_0^{(s)}(x)^q\, dx  \ge \int_{\mathbb{R}^n} f_0^\circ\, dx. 
$
To see this, we notice that for any $x \in \mathbb{R}^n$, 
\begin{align*}
    F_0^{(s)}(x) 
    &\le 
    \big( \int_{\mathbb{R}^n} f_0\, dz \big) \sup_{z\in \mathbb{R}^n} e^{\frac1p \langle x,z\rangle} f_0(z)^{\frac{1}p - 1}
    = 
    \big( \int_{\mathbb{R}^n} f_0\, dz \big) f_0^\circ( \frac{x}{1-p} )^{\frac1q}
\end{align*}
since $q = p'$. 
In view of $\int f_0\, dz \in (0,\infty)$, this and the change of variable yield that 
$$
\int_{\mathbb{R}^n} F_0^{(s)}(x)^q\, dx\ge 
\big( \int_{\mathbb{R}^n} f_0\, dz\big)^q  (1-p)^n  \int_{\mathbb{R}^n}  f_0^\circ \, dx 
\to 
\int_{\mathbb{R}^n}  f_0^\circ \, dx,\;\; s \to 0 
$$
as we wished. 
In any case, we obtain from Fatou's lemma and Theorem \ref{t:MonoHC} that 
$$
\int_{\mathbb{R}^n} f_{t_1}^\circ \, dx 
\le
\liminf_{s\to0} e^{ Q_s(t_1) } 
\le 
\liminf_{s\to0} e^{ Q_s(t_2) }
$$
for $0\le t_1 < t_2$.
In view of \eqref{e:CreatPolar}, if we formally interchange the order of limit and integral, then we obtain 
$\liminf_{s\to0} e^{ Q_s(t_2) }
=
\int_{\mathbb{R}^n} f_{t_2}^\circ \, dx 
$
which concludes $v(f_{t_1})\le v(f_{t_2})$. 
Therefore, we have only to confirm that 
\begin{equation}\label{e:Changelimint}
\lim_{s\to0} \int_{\mathbb{R}^n} F_t^{(s)}(x)^q\, dx = 
\int_{\mathbb{R}^n} f_t^{\circ}\, dx 
\end{equation}
holds as long as $t>0$. 
To see this, we notice from the regularization of $P_t^*$ that $f_t \ge c_t {\mathbf 1}_{[-l_t,l_t]^n}$ for some $c_t,l_t>0$, where ${\mathbf 1}_E(x) = 1$ if $x\in E$ and $=0$ if $x\notin E$ for general measurable set $E$. This reveals that 
$$
F_t^{(s)}(x) \ge c_t^\frac1p \int_{[-l_t,l_t]^n} e^{\frac1p\langle x,z\rangle}\, dz = c_t^\frac1p p^n \prod_{i=1}^n \frac{e^{ \frac{l_t}{p}x_i } - e^{ -\frac{l_t}{p}x_i }}{x_i}.  
$$
By using an elementary inequality $\frac1t ( e^{\alpha \xi} - e^{-\alpha \xi} ) \ge \frac12 \alpha e^{\frac12 \alpha |\xi|}$ for $\xi\in \mathbb{R}$ and $\alpha\ge0$, we obtain that 
\begin{align}\label{e:DomiFunc}
F_t^{(s)}(x)^q
&\le 
c_t^{-e^{2s}} \big( \frac{l_t}2 \big)^{n(1-e^{2s})} e^{-\frac{l_t}2 e^{2s} \|x\|_{\ell^1}} 
\sim 
c_t^{-1}  e^{-\frac{l_t}2  \|x\|_{\ell^1}}\;\;\;s\to0. 
\end{align}
Since $ c_t^{-1}  e^{-\frac{l_t}2  \|x\|_{\ell^1}} \in L^1(dx)$, we may apply Lebesgue's convergence theorem to conclude \eqref{e:Changelimint}.

\end{proof}

\begin{proof}[Proof of Theorem \ref{t:SymmRevHC}]
From H\"{o}lder's inequality, we have only to show 
$$
\big\| P_s \big[ \big( \frac{f_0}{\gamma} \big)^\frac1p \big] \big\|_{L^q(\gamma)} \ge \big( \int_{\mathbb{R}^n} f_0\, dx \big)^\frac1p
$$
for $p=1 - e^{-2s}$ and $q=1-e^{2s}$. 
In view of $q<0$, this is a consequence of Theorem \ref{t:MonoHC}, $\lim_{t\to\infty}f_t(x) =  (\int_{\mathbb{R}^n}f_0\, dx)\gamma(x)$ and Fatou's lemma as  
$$
\big\| P_s \big[ \big( \frac{f_0}{\gamma} \big)^\frac1p \big] \big\|_{L^q(\gamma)}^q \le {\liminf_{t\to\infty}} \big\| P_s \big[ \big( \frac{f_t}{\gamma} \big)^\frac1p \big] \big\|_{L^q(\gamma)}^q \le \big( \int_{\mathbb{R}^n} f_0\, dx \big)^\frac{q}p. 
$$
\end{proof}

In below, we prove Theorem \ref{t:MonoHC}. 
A key ingredient is the Brascamp--Lieb inequality refining the Poincar\'{e} inequality for a log-concave measure. 
\begin{theorem}[Brascamp--Lieb \cite{BraLi_JFA}] \label{t:BraLi}
    Let $h \in C^2(\mathbb{R}^n) \cap L^1(\mathbb{R}^n)$ be nonnegative and strictly log-concave. 
    Then for any locally Lipschitz $g \in L^2(hdx)$, we have that 
    \begin{equation}\label{e:P-BL}
        \int_{\mathbb{R}^n} |g|^2\, \frac{h}{m(h)}dx - \big(\int_{\mathbb{R}^n} g\, \frac{h}{m(h)}dx \big)^2 \le \int_{\mathbb{R}^n} \big\langle \nabla g, \big( \nabla^2 (-\log\, h) \big)^{-1} \nabla g \big\rangle \, \frac{h}{m(h)}dx, 
    \end{equation}
    where $m(h):=\int_{\mathbb{R}^n} \, h\,dx$. 
\end{theorem}

\begin{proof}[Proof of Theorem \ref{t:MonoHC}]
In this proof, we fix $p = p_s = 1-e^{-2s}$ and $q=q_s=1-e^{2s}$. 
We shall first show that $Q_s(t_1)\le Q_s(t_2)$ for any  $0<t_1\le t_2$, and then treat the case $t_1=0$ later. Since $t_1,t_2>0$ are now fixed, without loss of generality, we may suppose that $f_0$ is bounded and compactly supported by the standard approximation argument. 
This is because, for a fixed $t>0$ and an arbitrary $f_0$, we may see that 
\begin{equation}\label{e:Change12Oct}
    \lim_{N\to \infty} \int_{\mathbb{R}^n} \big( \int_{\mathbb{R}^n} e^{\frac1p\langle x,z\rangle} \big(f_0^{(N)}\big)_t(z)^\frac1p\, dz \big)^q\, dx 
=
e^{Q_s(t)}, 
\end{equation}
where $f_0^{(N)}:= f_0 {\mathbf 1}_{[-N,N]^n \cap \{ f_0 \le N\}}$ by virtue of the regularization of $P_t^*$ and \eqref{e:DomiFunc}. 
From this approximation and the explicit expression of ${P_t^*}$, we may check that  
\begin{equation}\label{e:CompSupp}
c_1 (1+|x|+|x|^2)\gamma_\beta \le f_t, |\nabla f_t|, |\Delta f_t| \le c_2 (1+|x|+|x|^2) \gamma
\end{equation}
for some $c_1, c_2, \beta>0$ depending on $f_0, t$ as long as $t>0$.  
The pointwise bound \eqref{e:CompSupp} is enough to justify applications of  Lebesgue's  convergence theorem and integration by parts, that we will use in the following argument.
In particular, \eqref{e:CompSupp} confirms that $Q_s'(t)$ is well-defined for all $t>0$, and so the goal is reduced to show $Q_s'(t)\ge0$ for $t>0$.

Since $f_t$ solves the Fokker--Planck equation $\partial_t f_t = \mathcal{L}^* f_t$, we have that 
\begin{align*}
-\frac{p}{q} Q_s'(t)
&=
- \frac{1}{ m(F_t^q) }
\bigg(
\int_{\R^n}
F_t(x)^{q-1}
\left(
\int_{\R^n} \mathcal{L}_z[e^{\frac{1}{p} \langle x, z \rangle } f_t(z)^{\frac1p-1}] f_t(z) \, dz
\right)
\, dx
\bigg), 
\end{align*}
where $F_t$ is defined in \eqref{e:F_t}. 
By using $\Delta f_t = f_t \Delta \log f_t + f_t |\nabla \log f_t |^2$, we notice that 
\begin{align*}
& \mathcal{L}_z[e^{\frac{1}{p} \langle x, z \rangle } f_t(z)^{\frac1p-1}]
\\
&=
\frac{|x|^2}{p^2} e^{\frac1p \langle x, z \rangle} f_t(z)^{\frac1p-1} 
+ \frac{2}{p}(\frac1p -1) e^{\frac1p \langle x, z \rangle} f_t(z)^{\frac1p-1} \langle x, \nabla \log f_t(z) \rangle \\
& \quad + (\frac1p -1)^2 e^{\frac1p \langle x, z \rangle} f_t(z)^{\frac1p -1} |\nabla \log f_t(z)|^2 
 + (\frac1p-1) e^{\frac1p \langle x, z \rangle} f_t(z)^{\frac1p-2} \Delta \log f_t(z) \\
&\quad - \frac1p \langle x, z \rangle e^{\frac1p \langle x, z \rangle} f_t(z)^{\frac1p-1}
- (\frac1p -1) e^{\frac1p \langle x, z \rangle} f_t(z)^{\frac1p-1} \langle z, \nabla \log f_t(z) \rangle, 
\end{align*}
and hence 
\begin{align*}
-\frac{p}{q} m(F_t^q) Q_s'(t)
&=
- \frac{1}{p^2}
\int_{\R^n}
|x|^2 F_t(x)^{q}
\, dx
\\
&\quad 
- \frac{2}{p}(\frac1p -1) 
\int_{\R^n}
F_t(x)^{q-1}
\left\langle
x, 
\int_{\R^n} e^{\frac1p \langle x, z \rangle } f_t^{\frac1p} \nabla \log f_t \, dz
\right\rangle
\, dx
\\
& \quad 
- (\frac1p -1)^2
\int_{\R^n}
F_t(x)^{q-1}
\left(
\int_{\R^n} e^{\frac1p \langle x, z \rangle } f_t^{\frac1p } |\nabla \log f_t|^2  \, dz
\right)
\, dx
\\
& \quad 
- (\frac1p-1)
\int_{\R^n}
F_t(x)^{q-1}
\left(
\int_{\R^n} e^{\frac1p \langle x, z \rangle } f_t^{\frac1p} \Delta \log f_t  \, dz
\right)
\, dx
\\
& \quad 
+ \frac1p
\int_{\R^n}
F_t(x)^{q-1}
\left\langle
x, 
\int_{\R^n} z e^{\frac1p \langle x, z \rangle} f_t^{\frac1p}  \, dz
\right\rangle
\, dx
\\
& \quad 
+ (\frac1p -1)
\int_{\R^n}
F_t(x)^{q-1}
\left(
\int_{\R^n} e^{\frac1p \langle x, z \rangle } f_t^{\frac1p} \langle z, \nabla \log f_t \rangle  \, dz
\right)
\, dx. 
\end{align*}
By $\nabla \log f_t = \frac{1}{f_t}\nabla f_t$ and  integration by parts, we may check the following identities 
\begin{align*}
&\int_{\R^n} e^{\frac1p \langle x, z \rangle} f_t^{\frac1p} \langle z, \nabla \log f_t \rangle \, dz
=
- pn {F_t(x)} 
- \left\langle 
x,  
\int_{\R^n} z e^{\frac1p \langle x, z \rangle} f_t^{\frac1p} \, dz
\right\rangle, 
\\
&\int_{\R^n} e^{\frac1p \langle x, z \rangle} f_t^{\frac1p} \nabla \log f_t \, dz
=
- x { F_t(x)},  
\\
&\int_{\R^n} e^{\frac1p \langle x, z \rangle} f_t^{\frac1p } |\nabla \log f_t|^2  \, dz
=
|x|^2 { F_t(x)} 
- p\int_{\R^n} e^{\frac1p \langle x, z \rangle} f_t^{\frac1p } \Delta \log f_t  \, dz. 
\end{align*}
Applying these identities, it follows that 
\begin{align*}
-\frac{p}{q} m(F_t^q) Q_s'(t)
&=
- 
\int_{\R^n}
|x|^2
F_t(x)^{q}
\, dx
\\
& \quad 
- (1-p) 
\int_{\R^n}
F_t(x)^{q-1}
\left(
\int_{\R^n} e^{\frac1p \langle x, z \rangle} f_t^{\frac1p } \Delta \log f_t  \, dz
\right)
\, dx
\\
& \quad 
+
\int_{\R^n}
F_t(x)^{q-1}
\left\langle 
x, 
\int_{\R^n} z e^{\frac1p \langle x, z \rangle} f_t^{\frac1p} \, dz
\right\rangle
\, dx
\\
& \quad 
- n(1-p)
\int_{\R^n}
F_t(x)^{q}
\, dx. 
\end{align*}
For the third term, we notice that 
\begin{align*}
\int_{\R^n}
F_t(x)^{q-1}
\left\langle 
x, 
\int_{\R^n} z e^{\frac1p \langle x, z \rangle} f_t^{\frac1p} \, dz
\right\rangle
\, dx 
&=
p
\int_{\R^n}
F_t(x)^{q-1}
\left\langle 
x, 
\nabla
F_t(x)
\right\rangle
\, dx 
\\ 
&=
-\frac pq n
\int_{\R^n}
F_t(x)^{q}
\, dx. 
\end{align*}
Therefore, together with $\frac1p + \frac1q =1$, we obtain the identity that 
\begin{align}
-\frac{p}{q}  Q_s'(t)
&=
- 
\int_{\R^n}
|x|^2
\frac{F_t(x)^{q}}{m(F_t^q)}
\, dx \label{e:STep1_Id}
\\ 
& \quad 
- (1-p) 
\int_{\R^n}
\frac{F_t(x)^{q-1}}{m(F_t^q)}
\left(
\int_{\R^n} e^{\frac1p \langle x, z \rangle} f_t^{\frac1p } \Delta \log f_t  \, dz
\right)
\, dx. \nonumber 
\end{align}
We now apply the Poincar\'{e}--Brascamp--Lieb inequality \eqref{e:P-BL} to the first term. 
To this end, we remark that $F_t^q$ is even and strictly log-concave. 
The evenness is an immediate from the evenness of $f_0$. To see the strict log-concavity, we observe that 
\begin{align*}
\nabla^2 (-\log\, F_t^q)(x)
&= 
-\frac{q}{p^2} {\rm cov}\, (h_{t,x}),
\end{align*}
where 
$$
h_{t,x}(z):= e^{ \frac1p\langle x,z\rangle } f_t(z)^\frac1p \frac1{F_t(x)},
$$
and 
\begin{align*}
{\rm cov}\, (h)
&:=
\int_{\mathbb{R}^n} z\otimes z \frac{h(z)}{m(h)}\, dz
-
\big(
\int_{\mathbb{R}^n} z \frac{h(z)}{m(h)}\, dz
\big)
\otimes
\big(
\int_{\mathbb{R}^n} z \frac{h(z)}{m(h)}\, dz
\big). 
\end{align*}
Since ${\rm cov}\, (h_{t,x})>0$, this confirms the strict log-concavity of $F_t^q$. 
Since  $F_t^q$ is even,  it holds that $\int_{\mathbb{R}^n} x_i F_t(x)^q\, dx = 0$ for $i=1,\ldots,n$. Hence we may apply \eqref{e:P-BL} with $g(x)=x_i$ and $h = F_t^q$ for each $i=1,\ldots,n$, and sum up to see that 
$$
\int_{\R^n}
|x|^2
\frac{F_t(x)^{q}}{m(F_t^q)} 
\, dx
\le 
\int_{\R^n}
{\rm Tr}
\left[
\left(
\nabla^2  (- \log F_t^q) (x)
\right)^{-1}
\right]
\frac{F_t(x)^q}{m(F_t^q)}
\, dx.
$$
We then appeal to the matrix form of  Cram\'{e}r--Rao's inequality; see \cite[(12)]{ELS} for instance, to estimate 
\begin{align*}
\big( \nabla^2 (-\log\, F_t^q)(x) \big)^{-1}
&= 
-\frac{p^2}{q} {\rm cov}\, (h_{t,x})^{-1} 
\le 
-\frac{p^2}{q} \int_{\mathbb{R}^n} \nabla^2_z ( -\log\, h_{t,x} )(z) h_{t,x}(z)\, dz. \end{align*}
This yields that 
$$
\int_{\R^n}
|x|^2
\frac{F_t(x)^{q}}{m(F_t^q)} 
\, dx
\le 
-\frac{p}{q} 
\int_{\R^n}
\big(
\int_{\mathbb{R}^n} \Delta ( -\log\, f_t )(z)h_{t,x}(z)\, dz
\big)
\frac{F_t(x)^q}{m(F_t^q)}
\, dx.
$$
By putting together this with \eqref{e:STep1_Id} and $ \frac{p}q = p-1 $, we conclude $Q_s'(t)\ge0$ for $t>0$.

We next consider the case $t_1=0$. We have only to show  that $Q_s(0) \le Q_s(t_0)$ for all $t_0>0$. 
Fix $t_0>0$ and an arbitrary nonzero $f_0$. Let 
$$
f_0^{(N)}:= f_0 \mathbf{1}_{[-N,N]^n},\;  F_0^{(N)}(x):= \int_{\mathbb{R}^n} e^{\frac1p \langle x,z\rangle} f_0^{(N)}(z)^\frac1p\, dz.
$$ 
Since $f^{(N)}_0 \uparrow f_0$, the monotone convergence theorem confirms that $F_0^{(N)} \uparrow F_0$. 
So Fatou's lemma shows that 
\begin{align*}
    e^{Q_s(0)}
    &= 
    \int_{\mathbb{R}^n} 
    \lim_{N\to\infty} 
    F_0^{(N)}(x)^q\, dx 
    \le 
    \liminf_{N\to \infty} 
    \int_{\mathbb{R}^n} 
    F_0^{(N)}(x)^q\, dx. 
\end{align*}
Next we show that 
\begin{equation}\label{e:Fixed-N}
\int_{\mathbb{R}^n} 
F_0^{(N)}(x)^q\, dx 
\le 
\int_{\mathbb{R}^n}
\big(
\int_{\mathbb{R}^n} 
e^{\frac1p \langle x,z\rangle} \big( f_0^{(N)} \big)_{t_0}(z)^\frac1p\, dz 
\big)^q\, dx
\end{equation}
for each fixed $N$. 
To see this, we note that 
$$
F_0^{(N)}(x)
=
\lim_{t\to0} 
\int_{\mathbb{R}^n} 
e^{\frac1p \langle x,z\rangle} \big( f_0^{(N)} \big)_{t}(z)^\frac1p\, dz 
$$
holds for each $x$ as $f_0^{(N)}$ is compactly supported.
Since we have already proved $Q_s(t)\le Q_s(t_0)$ for $t\in (0,t_0)$, this together with Fatou's lemma yields that 
\begin{align*}
\int_{\mathbb{R}^n} 
F_0^{(N)}(x)^q\, dx 
&\le 
\liminf_{t\to0} 
\int_{\mathbb{R}^n}
\big(
\int_{\mathbb{R}^n} 
e^{\frac1p \langle x,z\rangle} \big( f_0^{(N)} \big)_{t}(z)^\frac1p\, dz 
\big)^q\, dx\\
&\le 
\int_{\mathbb{R}^n}
\big(
\int_{\mathbb{R}^n} 
e^{\frac1p \langle x,z\rangle} \big( f_0^{(N)} \big)_{t_0}(z)^\frac1p\, dz 
\big)^q\, dx
\end{align*}
which is \eqref{e:Fixed-N}. 
Therefore, we conclude from \eqref{e:Change12Oct} at $t=t_0$ that 
$$
e^{Q_s(0)}
\le 
\liminf_{N\to\infty} 
\int_{\mathbb{R}^n}
\big(\int_{\mathbb{R}^n} 
e^{\frac1p \langle x,z\rangle} \big( f_0^{(N)} \big)_{t_0}(z)^\frac1p\, dz\big)^q
\, dx
=e^{Q_s(t_0)}.
$$
\end{proof}

\section{Concluding remarks}\label{SecRemark}
\subsection{$L^r$-volume product by Berndtsson--Mastrantonis--Rubinstein}
Berndtsson--Mastrantonis--Rubinstein \cite{BMR} introduced the $L^r$-volume product\footnote{In their original paper, they introduced a name of \textit{$L^p$-Mahler volume}. In order to avoid a potential confusion, we use the terminology $L^r$-volume product here in stead of $L^p$-Mahler volume. } which is defined as 
$$
\mathcal{M}_r(K) \coloneqq |K| \int_{\R^n} \big( \int_K e^{r \langle x, y \rangle} \frac{dy}{|K|} \big)^{-\frac1r}\, dx, 
$$
for $r>0$ and a convex body $K \subset \R^n$. 
One may see that $\lim_{r\to \infty} \mathcal{M}_r(K) = n! v(K)$. 
For this $L^r$-volume product, Berndtsson--Mastrantonis--Rubinstein \cite{BMR} established that the Blaschke--Santal\'{o} type inequality  
\begin{equation}\label{e:M_pBS}
\mathcal{M}_r(K) \le \mathcal{M}_r(\mathbf{B}_2^n)
\end{equation}
holds for any symmetric convex body $K \subset \R^n$ and all $r>0$. 
This $L^r$-volume product may be realized in the framework of our reverse hypercontractivity/Laplace transform. Let ${p}>0$ and $r>0$ be arbitrary. Then one may see from the simple change of variables that 
$$
\mathcal{M}_r(K)
= 
r^{-n}
\| f_K\|_{L^{{p}}(dx)}^{- p' }
\big\| \mathfrak{L} f_K \big\|_{L^{-\frac1r}(dx)}^{-\frac1r} 
$$
for any  convex body $K\subset \mathbb{R}^n$, where $f_K(x) := \frac1{|K|} \mathbf{1}_K(x)$. One has only to take $q=p'$ and $r=-\frac1q$ to compare Corollary \ref{Cor:Laplace}. Indeed, with this choice, \eqref{e:M_pBS} may be read as 
\begin{equation}\label{e:M_pLaplace}
    \big\|\mathfrak{L} f_K \big\|_{L^q(dx)} \ge 
    \frac{ \big\|\mathfrak{L} f_{\mathbf{B}^n_2} \big\|_{L^q(dx)} }{ \| f_{\mathbf{B}^n_2}\|_{L^p(dx)} } \|f_K\|_{L^p(dx)}. 
\end{equation}
Given this, one may wonder some relation between the $L^r$-volume product and the Laplace transform which is the analogue to the one between the classical volume product and the functional volume product.
For the classical volume product, we have a clear relation $v(e^{-\frac12\|\cdot\|_K^2})= (2\pi)^n |\mathbf{B}^n_2|^{-2}v(K)$ from which one may rederive the classical Blaschke--Santal\'{o} inequality from the functional one. Conversely, one may derive the functional one for log-concave functions from the classical one (for all dimension), see \cite{AKM}. 
However, in the frame work of the Laplace transform and $L^r$-volume product, the relation is less clear because of the lack of the duality. 
For instance, it is not obvious\footnote{If one takes $f=f_K$ in Corollary \ref{Cor:Laplace}, one would obtain some inequality similar to \eqref{e:M_pLaplace}. However, the constant of the inequality does not match. Indeed, \eqref{e:M_pLaplace} is stronger in such special cases. } to us if one may recover \eqref{e:M_pBS} or equivalently \eqref{e:M_pLaplace} from Corollary \ref{Cor:Laplace} with a choice $f(x)=e^{-\frac12\|x\|_K^2}$. Similarly, a simple adaptation of the lifting argument as in \cite{AKM} to this framework does not work well\footnote{If one applies the lifting argument to this frame work in order to deduce Corollary \ref{Cor:Laplace} (for log-concave functions) from \eqref{e:M_pBS}, one eventually faces to an issue on the forward  Minkowski's integral inequality for $L^q$-norm with $q<0$ which is not true in general.}. 
Therefore, there is no implication relation (at least not in a direct or obvious way) between the statement of Corollary \ref{Cor:Laplace} and \eqref{e:M_pBS}

\subsection{Brascamp--Lieb theory view point and Kolesnikov--Werner's conjecture}
We here explain our results from a view point of the Brascamp--Lieb theory, and point out a link to Kolesnikov--Werner's conjecture \cite{KW}.  
Let us recall that ${\rm BL}_s^{(e)}\ge0$ denotes the best constant of \eqref{e:HC-IBL} for even functions. 
As a corollary of Theorem \ref{t:SymmRevHC} and the duality argument, we obtain the following: 
\begin{corollary}\label{Cor:InvBL}
    Let $s>0$, $1-e^{2s}\le q_s<0<p_s \le 1-e^{-2s}$, and $c_1(s)=\frac1{p_s},c_2(s)=\frac1{q_s'}$. Then for $\mathcal{Q}_s$ given by \eqref{e:ExampleBLdata}, ${\rm BL}_s^{(e)}$ is exhausted by centered Gaussians, that is  
    $$
    {\rm BL}_s^{(e)} = \inf_{A_1,A_2>0} \int_{\mathbb{R}^{2n}} e^{-\pi \langle x, \mathcal{Q}_s x\rangle} \prod_{i=1,2} \gamma_{A_i}(x_i)^{c_i(s)}\, dx.  
    $$
\end{corollary}
We emphasize that this corollary does not follow from the general result due to Barthe--Wolff \cite{BW}. 
To clarify the situation, let us recall their result.
Let $m$, $d,d_1,\ldots, d_m \in \mathbb{N}$, $c_1,\ldots,c_m \in \mathbb{R}\setminus \{0\}$, $L_i:\mathbb{R}^d\to \mathbb{R}^{d_i}$ be a linear surjective map for $i=1,\ldots,m$, and $\mathcal{Q}$ be a self-adjoint matrix on $\mathbb{R}^d$. We often abbreviate $\mathbf{c} = (c_i)_{i=1}^m$ and $\mathbf{L}= (L_i)_{i=1}^m$. 
Consider the inequalities of the form  
\begin{equation}\label{e:IBL}
\Lambda(\mathbf{L},\mathbf{c},\mathcal{Q};\mathbf{f}):={\int_{\mathbb{R}^d} e^{-\pi \langle x,\mathcal{Q}x\rangle} \prod_{i=1}^m f_i(L_ix)^{c_i}\, dx }
\ge 
C{\prod_{i=1}^m \big(\int_{\mathbb{R}^{d_i}} f_i\, dx_i\big)^{c_i} }
\end{equation}
for some $C\ge0$ and all $f_i$ in some appropriate class. 
Barthe--Wolff \cite{BW} considered the inequality \eqref{e:IBL} for \textit{all} nonnegative $f_i \in L^1(\mathbb{R}^{d_i})$, and established the analogue to Lieb's fundamental theorem \cite{Lieb} under some non-degeneracy condition. 
In order to state their condition, we need further notations. We order $(c_i)_i$ so that $c_1,\ldots, c_{m_+} >0 > c_{m_++1},\ldots, c_m$ for some $0\le m_+\le m$. Correspondingly, let $\mathbf{L}_+:\mathbb{R}^d \ni x \mapsto (L_1x,\ldots, L_{m_+}x) \in \prod_{i=1}^{m_+} \mathbb{R}^{d_i}$. Finally let $s^+(\mathcal{Q})$ denote the number of positive eigenvalues of $\mathcal{Q}$. 
The main theorem in \cite{BW} states that if  a data $(\mathbf{L},\mathbf{c},\mathcal{Q})$ satisfies Barthe--Wolff's non-degeneracy condition 
    \begin{equation}\label{e:BW-Nondeg}
        \mathcal{Q}|_{{\rm Ker}\, \mathbf{L}_+}>0\;\;\;{\rm and}\;\;\; d\ge s^+(\mathcal{Q}) + \sum_{i=1}^{m_+} d_i, 
    \end{equation}
    then the best constant in \eqref{e:IBL}  is exhausted by centered Gaussians, that is 
        $$
        \inf_{f_1,\ldots, f_m\ge0:\; \int f_i =1} \Lambda(\mathbf{L},\mathbf{c},\mathcal{Q}; \mathbf{f})
        = 
        \inf_{A_1,\ldots,A_m>0} \Lambda(\mathbf{L},\mathbf{c},\mathcal{Q}; \gamma_{A_1},\ldots, \gamma_{A_m}).  
        $$
Clearly for each $s>0$, our inequality \eqref{e:HC-IBL} is an example of this inverse Brascamp--Lieb inequality with the data $c_i = c_i(s)$, $L_i(x_1,x_2) = x_i$, and $\mathcal{Q}=\mathcal{Q}_s$.
However, the data does not satisfy the non-degeneracy condition \eqref{e:BW-Nondeg} when $p_s,q_s$ are beyond the Nelson's time regime $\frac{q_s-1}{p_s-1} > e^{2s}$. 
Given Corollary \ref{Cor:InvBL}, it seems to be reasonable to expect the following statement even when a data $(\mathbf{L},\mathbf{c},\mathcal{Q})$ is degenerate: 
for any data $(\mathbf{L},\mathbf{c},\mathcal{Q})$, it holds that 
\begin{equation}\label{e:SymIBL}
\inf_{f_1,\ldots, f_m\ge0:\; \int f_i=1,\; {\rm even}} \Lambda(\mathbf{L},\mathbf{c},\mathcal{Q}; \mathbf{f})
        = 
        \inf_{A_1,\ldots,A_m>0} \Lambda(\mathbf{L},\mathbf{c},\mathcal{Q}; \gamma_{A_1},\ldots, \gamma_{A_m}).  
\end{equation}
The importance of such a generalization may be seen by its link to the conjecture of Kolesnikov--Werner \cite{KW}, concerning an extension of the Blaschke--Santal\'{o} inequality to many convex bodies. 
\begin{conjecture}[Kolesnikov--Werner \cite{KW}]\label{Conj:KW}
    Let $m\ge2$. If nonnegative even functions $f_i \in L^1(\mathbb{R}^n)$ satisfy 
    \begin{equation}\label{e:CondKW}
    \prod_{i=1}^{m} f_i(x_i) \le \exp\big(  -\frac1{m-1} \sum_{1\le i <j \le m} \langle x_i,x_j \rangle \big),\;\;\; x_1,\ldots, x_m \in \mathbb{R}^n,
    \end{equation}
    then 
    \begin{equation}\label{e:ConsKW}
    \prod_{i=1}^m \big( \int_{\mathbb{R}^n} f_i\, dx_i \big)
    \le 
    \big( \int_{\mathbb{R}^n} e^{ -\frac12|x|^2 }\, dx \big)^m = (2\pi)^\frac{mn}2.
    \end{equation}
\end{conjecture}
To see the link to \eqref{e:SymIBL}, we take a specific data. 
Let $m\ge2$, $d_1,\cdots, d_m = n$, and $d = mn$. 
For each $s>0$, let  
$$
c_1 (s)= \cdots = c_m (s)= \frac1{1-e^{-2s}},\; L_i(x_1,\ldots, x_m) = x_i, \; 
\mathcal{Q}_s = - \kappa_{m,s}  
\big( \mathbbm{1} - {\rm id}_{\mathbb{R}^{mn}} \big),
$$
where $\kappa_{m,s}:=\frac{e^{-s}}{2\pi(m-1) (1-e^{-2s})}$, and $\mathbbm{1}$ denotes $mn\times mn$ matrix whose entries are all $1$. 
By following the argument in  Lemma  \ref{Prop:Obs1}, one may see that if the conjectural inverse Brascamp--Lieb inequality \eqref{e:SymIBL} could be true for such data, then it would yield the affirmative answer to Conjecture \ref{Conj:KW}.

\section*{Acknowledgements}
This work was supported by JSPS Overseas Research Fellowship and JSPS Kakenhi grant numbers 21K13806, 23K03156, and 23H01080 (Nakamura), and JSPS Kakenhi grant number 22J10002 (Tsuji). 
Authors would like to thank to Neal Bez for sharing his insight which leads us to this work. 
Authors are also grateful to organizers of the Online Asymptotic Geometric Analysis Seminar.
Authors were also benefited from comments from Dario Cordero-Erausquin. He pointed out to them the discussion of Klartag and Tao in Tao's blog post, and this improves the presentation of this paper. 


\end{document}